\documentclass{article}
\usepackage{amsthm}
\usepackage{graphicx}
\usepackage{amssymb}
\newcommand{\R}{\mathbb{R}}
\begin{document}

\title{The nonlinear Poisson equation via a Newton-imbedding procedure}

\author{Jonathan J. Sarhad}

\date{}

\maketitle

\begin{abstract}

This article considers the semilinear boundary value problem given by the Poisson equation, $-\Delta u=f(u)$ in a bounded domain $\Omega\subset \R^{n}$ with smooth boundary.  For the zero boundary value case, we approximate a solution using the Newton-imbedding procedure.  With the assumptions that $f$, $f'$, and $f''$ are bounded functions on $\R$, with $f'<0$, and $\Omega\subset \R^{3}$, the Newton-imbedding procedure yields a continuous solution.  This study is in response to an independent work which applies the same procedure, but assuming that $f'$ maps the Sobolev space $H^{1}(\Omega)$ to the space of H\"older continuous functions $C^{\alpha}(\bar{\Omega})$, and $f(u)$, $f'(u)$, and $f''(u)$ have uniform bounds.  In the first part of this article, we prove that these assumptions force $f$ to be a constant function.  In the remainder of the article, we prove the existence, uniqueness, and $H^{2}$-regularity in the linear elliptic problem given by each iteration of Newton's method.  We then use the regularity estimate to achieve convergence. 

\end{abstract}

\section*{0 Introduction}

The goal of this article is to find suitable hypotheses on a function $f\in C^{2}(\R)$ related to attaining a solution to the semilinear boundary value problem given by
\[(*)
\left\{ \begin{array}{rcl}
-\Delta u&=&f(u)   \hspace{5mm} \mbox{ in } \Omega\\
u|_{\Gamma}&=& \phi \hspace{5mm} \mbox{ on } \Gamma=\partial \Omega,
\end{array}
\right.  
\]

\noindent using the Newton-imbedding procedure that is applied in \cite{Hsiao}.  Here, $f(u)$ is defined as $f\circ u$. \textsl{In this sense $f$ can be viewed as a map from a space of real-valued functions to another space of real valued functions via composition.} In addition, $H^{k}(\Omega)$ is defined as the $L^{2}$ functions on $\Omega$ having (weak) $i^{th}$ derivatives ($1\leq |i|\leq k$) which are $L^{2}$ functions on $\Omega$. This is the Hilbert space notation substituted for the Sobolev space notation $W^{k,2}(\Omega)$.  The space of real-valued functions on $\Omega$ which are H\"older continuous with exponent $\alpha$ will be denoted $C^{\alpha}(\bar{\Omega})$.  The author of \cite{Hsiao} achieves an $H^{2}$ solution when $\Omega$ is a domain in $\R^{3}$ and $\Gamma$ is smooth, provided the following assumptions on $f$ hold:

\begin{itemize}

\item 1. $f$  is a continuous map from  $H^{2}(\Omega)$ to $L^{2}(\Omega).$
\item 2. $f'$ and $f''$  are continuous maps from  $H^{1}(\Omega)$ to $C^{\alpha}(\bar{\Omega})$, $\alpha\in (0,\frac{1}{2}]$.
\item 3. There exists a constant $M>0$ such that

 \[||f(u)||_{L^{2}(\Omega)}\leq M \mbox{ for all } u\in H^{2}(\Omega), \hspace{3mm} ||f'(u)||_{C^{\alpha}(\bar{\Omega})}\leq M \mbox{ for all } u\in H^{1}(\Omega),\hspace{3mm}  \hspace{3mm} \]
 
 \[\mbox{   and   } ||f''(u)||_{C^{\alpha}(\bar{\Omega})}\leq M \mbox{ for all } u\in H^{1}(\Omega).\]

\item 4. $(-f')$ is positive in the sense that $(-f'(u)v,v)>0$ for all $0\neq v\in H^{2}(\Omega)$

\end{itemize}
  
\noindent An additional condition in \cite{Hsiao} is the choice of a \textsl{uniform width} of time intervals in the procedure that ensures convergence, which exists as a consequence of the above assumptions. However, we prove the following theorems in Sections 2 and 3 of this article:
 
\newtheorem*{thm}{Theorem} \begin{thm} {\bf 2.1} \textsl{If $f:\R\rightarrow{}\R$ is a map from $H^{1}(\Omega)$ to $C^{0}(\bar{\Omega})$ via composition and $\Omega$ is a domain in $\R^{n}$ with $n>2$, then $f$ is a constant function.} 
\end{thm}

\begin{thm}{\bf 3.1} \textsl{Let $h:\R\rightarrow \R$ map $H^{2}(\Omega)\bigcap H_{0}^{1}(\Omega)$ to $L^{p}(\Omega)$ via composition, where $1\leq p\leq \infty$ and $\Omega$ is domain in $R^{n}$.  If there exists a constant $M>0$ such that $||h(u)||_{L^{p}}\leq M$ for all $u\in H^{2}(\Omega)\bigcap H_{0}^{1}(\Omega)$, then $h$ is a bounded function on $\R$, i.e. there exists a constant $C>0$ such that $|h(x)|\leq C$ for all $x\in \R$.}
\end{thm}

By Theorem 2.1, the assumption in (2) that $f'$ maps $H^{1}$ to $C^{\alpha}$ forces $f'$ to be a constant function.  Theorem 3.1 shows that the uniform bound on $f(u)$ in assumption (3) forces $f$ to be a bounded function on $\R$.  Thus $f$ is shown to be linear and bounded on $\R$, and is therefore a constant function, reducing the scope of the procedure in \cite{Hsiao} to the family of problems given by $-\Delta u=const.$ \\

In Section 1 of this article, we construct a `$mesa$' function (see Figure 1 in Section 1) whose existence in $H^{1}(\Omega)$ will serve as a counterexample to a non-constant mapping. In Section 2, the mesa function is used to prove Theorem 2.1.  In Section 3, Theorem 3.1 is proven using a sequence of smooth `bump' functions in $H^{2}$.  As a consequence of this, the uniform bounds also imposed in (3) on $f'(u)$ and $f''(u)$ imply that $f'$ and $f''$ are also bounded functions on $\R$. In Section 4, we describe and apply the Newton-imbedding procedure to the case of $(*)$ with a zero boundary condition.  Of primary importance in the procedure is the following linear boundary value problem,

\[(**)
\left\{ \begin{array}{rcl}
-\Delta u + q(x)u&=&g(x)   \hspace{5mm} \mbox{ in } \Omega\\
u|_{\Gamma}&=& 0  \hspace{5mm} \mbox{ on } \Gamma,
\end{array}
\right.
\]

\noindent given by each iteration in the Newton-imbedding procedure. Here, $q(x)$ is a positive scaling of $(-f')$ while $g(x)$ depends on $f$ and $f'$ in a manner that allows $g\in L^{2}$ under our assumptions. The exact hypotheses on $q$ and $g$ will be made precise in Section 4. As in \cite{Hsiao},  the assumption that $q>0$ allows for existence and uniqueness for $(**)$ in $H^{1}$, as well as the regularity lifting of the $H^{1}$ solution to $H^{2}$. For the remainder of the article, it will be understood that (**) is the general boundary value problem stated above, with the conditions that $g\in L^{2}$ and $q>0$.  Under the following assumptions on $f$,

\begin{itemize}

\item I.   $f$  is a continuous map from  $H^{2}(\Omega)$ to $L^{2}(\Omega).$
\item II.  $f'$ and $f''$ are continuous maps from  $H^{1}(\Omega)$ to $L^{n}(\Omega)$
\item III. there exists a constant $M>0$ such that

 \[|f|\leq M, \hspace{3mm} |f'|\leq M,\hspace{3mm}  \mbox{  and  } \hspace{3mm} |f''|\leq M.\]

\item IV.  $(-f')>0$,

\end{itemize}

\noindent we prove existence and uniqueness for (**) in Section 5, and achieve the regularity lifting of an $H_{0}^{1}$ solution of (**) to $H^{2}$ in Section 6.  These results are summarized in the following theorem: 

\begin{thm}{\bf 6.1}  \textsl{Let $\Omega$ be a bounded domain in $\R^{n}$ with smooth boundary $\Gamma$ and $n>2$. Then for $g\in L^{2}(\Omega)$, $q\in L^{n}(\Omega)$, and $q>0$, the linear boundary value problem}

\[(**)
\left\{ \begin{array}{rcl}
-\Delta u + q(x)u&=&g(x)   \hspace{5mm} \mbox{ in } \Omega\\
u|_{\Gamma}&=& 0  \hspace{5mm} \mbox{ on } \Gamma,
\end{array}
\right.
\]

\noindent \textsl{has a unique solution $u\in H^{2}(\Omega)\bigcap H_{0}^{1}$ with} 

\[||u||_{H^{2}(\Omega)} \leq C(||g||_{L^{2}(\Omega)}),\]

\noindent \textsl{where $C$ depends only on $\Omega$, $n$, and $q$.}

\end{thm}

In Section 7, under an additional assumption (V) concerning the uniform width of time intervals in the procedure, convergence in the procedure is achieved resulting in the following theorem:  

\begin{thm}{\bf 7.1} \textsl{With $\Omega$ a bounded domain in $\R^{3}$ with smooth boundary and assumptions (I)-(V), the semilinear boundary value problem},

\[(*')
\left\{ \begin{array}{rcl}
-\Delta u&=&f(u)   \hspace{5mm} \mbox{ in } \Omega\\
u|_{\Gamma}&=& 0 \hspace{5mm} \mbox{ on } \Gamma=\partial \Omega,
\end{array}
\right.  
\]

\noindent has a unique solution in $H^{2}(\Omega)\bigcap H_{0}^{1}(\Omega)$, and hence a continuous solution, which can be approximated by the Newton-imbedding procedure.\\
\end{thm}

\section{The Mesa Function}

Let $\Omega$ be a domain in $\R^{n}$ with $n>2$ and let $c\in \Omega$.  Since the function will be radially symmetric about $c$, define $r=\mid x-c \mid$ for $x\in \Omega$, and $T>0$ such that $B(c,T)\subset\subset \Omega$, where $B(c,T)$ denotes the open ball of radius $T$ about $c$. Also let $a,b\in \R$ with $a<b$, and $\alpha \in (0,\frac{n-2}{2})$. In order to define the function, it is necessary to decompose the interval $[0,T]$ as follows:\\ 

\vspace{5mm}

If we let $\displaystyle{r_{1}^{+}=\frac{T}{2}}$, then there is an $s_{1}^{+}$ such that

\[\frac{1}{(s_{1}^{+})^{\alpha}}-\frac{1}{(r_{1}^{+})^{\alpha}}= b-a.\] 

In particular, $0<s_{1}^{+}<r_{1}^{+}$.  Setting $\displaystyle{s_{1}^{-}=\frac{s_{1}^{+}}{2}}$ allows for an $r_{1}^{-}$ such that

\[\frac{1}{(r_{1}^{-})^{\alpha}}-\frac{1}{(s_{1}^{-})^{\alpha}}= b-a.\]

In particular, $0<r_{1}^{-}<s_{1}^{-}$.  Continuing in this manner, set $\displaystyle{r_{m+1}^{+}=\frac{r_{m}^{-}}{2}}$.\\

 Note that $r_{m+1}^{+}>0$ for all m and $r_{m+1}^{+}$ goes to zero with $\displaystyle{\frac{1}{2^{m}}}$.\\

\vspace{10mm}

Using the above notation, let $U:\Omega\rightarrow{}$R be the radially symmetric piecewise function defined inductively by\\

\vspace{5mm}

$\displaystyle{U(r)=\left\{\begin{array}{lcl}

\vspace{3mm}

0 & , & r\geq T\\

\vspace{3mm}

(\frac{-2a}{T})r+ 2a & , & r_{1}^{+}\leq r\leq T\\

\vspace{3mm}

\frac{1}{r^{\alpha}}-\frac{1}{(r_{m}^{+})^{\alpha}}+a & , & s_{m}^{+}\leq r\leq r_{m}^{+}\\

\vspace{3mm}

b & , & s_{m}^{-}\leq r\leq s_{m}^{+}\\

\vspace{3mm}

b-(\frac{1}{r^{\alpha}}-\frac{1}{(s_{m}^{-})^{\alpha}}) & , & r_{m}^{-}\leq r\leq s_{m}^{-}\\

\vspace{3mm}

a & , & r_{m+1}^{+}\leq r \leq r_{m}^{-}.

\end{array}
\right.}$\\

We will call $U(r)$ a $mesa$ function with exponent $\alpha$.  Figure 1, below, is a sketch of a mesa function whose partition points have been altered to show more `mesas'.  \\  

\begin{figure}[h]
\includegraphics[scale=0.6]{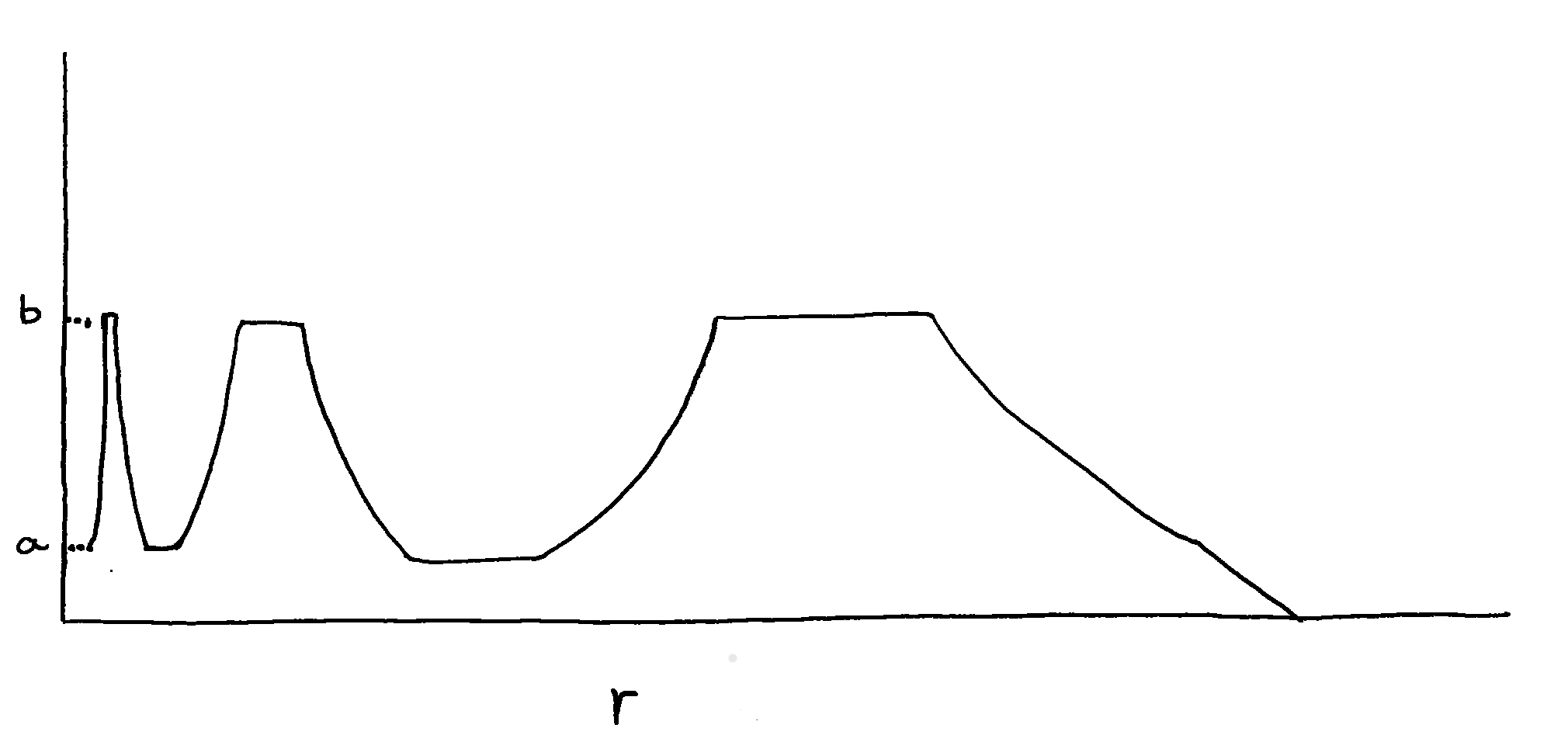}
\caption{Artist's depiction of a mesa function}
\end{figure}

$U$ is bounded and has compact support, so is trivially in $L^{2}(\Omega)$.  It remains to show that it has (weak) first derivatives in $L^{2}(\Omega)$. The proposed first derivatives are given by\\

$\displaystyle{U_{x_{i}}(r)=\left\{\begin{array}{lcl}

\vspace{3mm}

0 & , & r\geq T\\

\vspace{3mm}

\frac{-2a}{T} & , & r_{1}^{+}\leq r\leq T\\

\vspace{3mm}

\frac{-\alpha x_{i}}{r^{\alpha + 2}} & , & s_{m}^{+}\leq r\leq s_{m}^{+}\\

\vspace{3mm}

0 & , & s_{m}^{-}< r\leq s_{m}^{+}\\

\vspace{3mm}

\frac{\alpha x_{i}}{r^{\alpha + 2}} & , & r_{m}^{-}< r\leq s_{m}^{-}\\

\vspace{3mm}

0 & , & r_{m+1}^{+}< r\leq r_{m}^{-}.

\end{array}
\right.}$\\

\vspace{5mm}

Away from zero, on each annulus of the decomposed $\Omega$, the expressions in $U_{x_{i}}$ are classical derivatives of their corresponding expressions in $U(r)$. Let $\phi\in C_{0}^{\infty}(\Omega)$ and fix $N$.  Integrating $U\phi_{x_{i}}$ by parts over the annuli given by $[r_{1}^{+}, T]$, $[s_{m}^{+},r_{m}^{+}]$, $[s_{m}^{-}, s_{m}^{+}]$, $[r_{m}^{-}, s_{m}^{-}]$, and $[r_{m+1}^{+}, r_{m}^{-}]$ for $m=1,...,N$ and recalling that $U\equiv 0$ for $r\geq T$, gives\\

\vspace{5mm}

$\displaystyle{\int_{\Omega-B(c,r_{N+1}^{+})}{U\phi_{x_{i}}dx}}= -\int_{\Omega-B(c,r_{N+1}^{+})}{U_{x_{i}}\phi dx} + \int_{\partial B(c,r_{N+1}^{+})}{U\phi \rho^{i} dS}$,

\vspace{5mm}  

\noindent where $\rho=(\rho^{1},...,\rho^{n})$ is the inward pointing normal on $\partial B(c,r_{N+1}^{+})$.

\vspace{5mm}

Let $\displaystyle{u(r)=\frac{1}{r^{\alpha}}}$.  Note that $|U_{x_{i}}|\leq  |u_{x_{i}}|$, so that $|DU|\leq |Du|$.\\

\noindent Following the line of argument [1, p.246] given by L. Evans, since $\alpha< n-1$, $\displaystyle{|Du| = \frac{\alpha}{r^{\alpha +1}}\in L^{1}(\Omega)}$ and therefore $|DU|\in L^{1}(\Omega)$.\\

Letting $N\rightarrow \infty$ (and thus $r_{N+1}^{+}\rightarrow 0$),

\vspace{5mm} 

$\displaystyle{\left| \int_{\partial B(c,r_{N+1}^{+})}{U\phi \rho^{i} dS}\right|\hspace{1mm} \leq \hspace{1mm} \parallel U\phi \parallel_{\infty} \int_{\partial B(c,r_{N+1}^{+})}{\rho^{i} dS}}\hspace{1mm} \leq \hspace{1mm} M(r_{N+1}^{+})^{n-1}\rightarrow 0$,

\vspace{5mm}

hence \hspace{4mm} $\displaystyle{\int_{\Omega}{U\phi_{x_{i}}} dx = -\int_{\Omega}{U_{x_{i}}\phi dx}}.$

\vspace{5mm}

Therefore $U_{x_{i}}$ is a (weak) derivative of $U$.  Moreover, since $\displaystyle{\alpha< \frac{n-2}{2}}$,

\vspace{1mm}

\noindent following the arguement in [1, p.246], $\mid Du \mid\in L^{2}(\Omega)$ and thus $|DU|\in L^{2}(\Omega)$ and $U(r)\in H^{1}(\Omega)$. The following lemma summarizes the above discussion:\\
 
\newtheorem{lem}{Lemma 1.1} Lemma 1.1 \textsl{If $\Omega$ is a domain in $\R^{n}$ with $n>2$, and $U(r)$ is a mesa function with exponent $\displaystyle{\alpha< \frac{n-2}{2}}$, then $U(r)\in H^{1}(\Omega)$.}

\vspace{5mm}

\section{Constant Mapping}

\begin{thm} {\bf 2.1} \textsl{If $f:\R\rightarrow{}\R$ is a map from $H^{1}(\Omega)$ to $C^{0}(\bar{\Omega})$ via composition and $\Omega$ is a domain in $\R^{n}$ with $n>2$, then $f$ is a constant function.} 
\end{thm}

\begin{proof}

Suppose on the contrary, that $f$ is not constant and assumes distinct values at $a$ and $b$.  Without loss of generality, assume that $a< b$. Let $c\in\Omega$ and $T$ be such that $B(c,T)\subset\subset \Omega$.  Since $n>2$, there exists $\alpha$ such that $0<\alpha<\frac{n-2}{2}$.  Let $U(r)$ be the mesa function centered at c, with exponent $\alpha$, support in $B(c,T)$, and prescribed maximum and minimum, $b$ and $a$, respectively.  By the above lemma, $U(r)$ is in $H^{1}(\Omega)$.  Using the notation in the previous section for the domain of $U(r)$, it holds that for any $\delta > 0$ there exists an $N$ such that $[s_{N}^{-},s_{N}^{+}]\subset B(c,\delta)$ and $[r_{N+1}^{+},r_{N}^{-}]\subset B(c,\delta)$.  Note that $f\circ U \equiv f(b)$ on $[s_{N}^{-},s_{N}^{+}]$ and $f\circ U \equiv f(a)$ on $[r_{N+1}^{+},r_{N}^{-}]$.  Since the measure of the above intervals is strictly positive, $f\circ U$ has no continuous representative.  In other words, the oscillations of $f\circ U$ do not diminish in any neighborhood of $c$. This contradicts the hypothesis that $f$ maps $U$ to a continuous function.

\end{proof}

Now, as an immediate application of Theorem 2.1, the assumption in (2) that $f^{'}$ maps $H^{1}$ into continuous functions forces $f^{'}$ to be constant.\\

\section{Uniform Bounds}

For this Section we assume $\Omega$ is a domain in $\R^{n}$.\\

\begin{thm}{\bf 3.1} \textsl{Let $f:\R\rightarrow \R$ map $H^{2}(\Omega)\bigcap H_{0}^{1}(\Omega)$ to $L^{p}(\Omega)$ where $1\leq p\leq \infty$.  If there exists a constant $M>0$ such that $||f(u)||_{L^{p}}\leq M$ for all $u\in H^{2}(\Omega)\bigcap H_{0}^{1}(\Omega)$, then $f$ is a bounded function on $\R$, i.e. there exists a constant $C>0$ such that $|f(x)|\leq C$ for all $x\in \R$.}
\end{thm}

\begin{proof}  Let $p<\infty$.  Suppose on the contrary, that $f$ is not bounded.  Then there exists a sequence, $\{x_{k}\}_{k=1}^{\infty}$ in $\R$ such that $|f(x_{k})|>k$.  Let $y_{0}\in \Omega$ and $r$ such that $B=B(y_{0},r)\subset\subset \Omega$.  Set $B_{\frac{1}{2}}=B(y_{0},\frac{r}{2})$. Choose a smooth function, $\gamma$, such that $\gamma\equiv 1$ on $B_{\frac{1}{2}}$, $\gamma\equiv 0$ on $\Omega-B$, and $0\leq \gamma\leq 1$.  Define the smooth function $u_{k}$ on $\Omega$ by $u_{k}=x_{k}\gamma$.  Then $u_{k}\in H^{2}(\Omega)\bigcap H_{0}^{1}(\Omega)$ for all $k$ and 

\[||f(u_{k})||_{L^{p}(\Omega)}\geq ||f(u_{k})||_{L^{p}(B_{\frac{1}{2}})}=||f(x_{k})||_{L^{p}(B_{\frac{1}{2}})}>k|B_{\frac{1}{2}}|^{\frac{1}{p}}. \]

\noindent Choosing $k_{0}$ large enough such that $k_{0}|B_{\frac{1}{2}}|^{\frac{1}{p}}>M$ gives a contradiction.  If $p=\infty$, a similar computation holds, choosing $k_{0}>M$.

\end{proof}

Remark:  Since the $C^{\alpha}$ norm has the $L^{\infty}$ norm as a summand, Theorem 5.1 with $p=\infty$ suffices to show that a uniform bound on $||f(u)||_{C^{\alpha}}$ implies $f$ is bounded.  Therefore the assumptions made in \cite{Hsiao}, imply that $f$, $f'$, and $f''$ are bounded functions.  Moreover, under the same assumptions, as shown in the previous Section, $f$ is linear.  In this case $f$ is a constant, reducing the scope of the procedure to problems given by $-\Delta u=const.$

\section{Newton-imbedding Procedure}

The Newton-imbedding procedure we wish to apply to

\[(*')
\left\{ \begin{array}{rcl}
-\Delta u&=&f(u)   \hspace{5mm} \mbox{ in } \Omega\\
u|_{\Gamma}&=& 0 \hspace{5mm} \mbox{ on } \Gamma=\partial \Omega,
\end{array}
\right.  
\]

\noindent has two parts. It is well described in \cite{Hsiao}, but recalled here for clarity.  The  procedure first imbeds the problem in a one-parameter family of problems,

\[-\Delta u = tf(u)  \hspace{6mm} \mbox{  in  } \Omega\]

\noindent with $u=0$ on $\Gamma$ and parameter $t\in[0,1]$.  We set 

\[F_{t}(u)= \Delta u + tf(u).\] 

\noindent Solving $(*')$ is then a matter of solving $F_{1}(u)=0$. Let $u(x,t)$ be the solution to $F_{t}(u)=0$. Starting with $t_{0}=0$, the problem is solved with solution $u(x,0)$ in $\Omega$.  Observe that with boundary value zero imposed, $u(x,0)$ is uniquely determined as $u(x,0)\equiv 0$.  To solve $F_{t_{1}}(u)=0$, $u(x,0)$ is taken as an initial approximation and the standard Newton's method is applied.  With convergence, the solution $u(x,t_{1})$ to $F_{t_{1}}(u)=0$ is achieved.  The function $u(x,t_{1})$ is then used as an initial approximation for $F_{t_{2}}(u)=0$ and so on for increasing times $t_{j}$.  Thus the solutions are pushed along with increasing times using Newton's method with the goal of reaching $t=1$ in finitely many time shifts.  Let $u_{0}(x,t_{j})=u(x,t_{j-1})$, the initial approximation for $F_{t_{j}}(u)=0$ and $u_{m}(x,t_{j})$ be the $m^{th}$ iteration of Newton's method at time $t_{j}$.  In the following discussion, the argument of the $u_{m}$'s will be suppressed.  We will also temporarily use the symbol $D$ for the Frechet derivative in contrast to its usual use as the gradient.  Note that

\[DF_{t_{j}}(u_{m})[w]=\Delta w + t_{j}Df(u_{m})[w] \hspace{3mm} \mbox{ and } \hspace{3mm} Df(u_{m})[w]= f'(u_{m})w \]

\noindent for $w\in H^{2}(\Omega)$ and that the $(m+1)^{th}$ iterate in the Newton approximation is given by 

\[DF_{t_{j}}(u_{m})[u_{m+1}-u_{m}]= -F_{t_{j}}(u_{m}).\]

\noindent  In this case, the $(m+1)^{th}$ iteration at time $t_{j}$ yields the following linear problem:

\[(**)
\left\{ \begin{array}{rcl}
-\Delta u_{m+1} + (-t_{j}f'(u_{m}))(u_{m+1})&=&t_{j}(f(u_{m})-f'(u_{m})u_{m})   \hspace{5mm} \mbox{ in } \Omega\\
 u_{m+1}|_{\Gamma}&=& 0  \hspace{32mm} \mbox{ on } \Gamma.
\end{array}
\right.
\]

\noindent This is the problem

\[(**)
\left\{ \begin{array}{rcl}
-\Delta u + q(x)u&=&g(x)   \hspace{5mm} \mbox{ in } \Omega\\
u|_{\Gamma}&=& 0 \hspace{5mm} \mbox{ on } \Gamma,
\end{array}
\right.
\]

\noindent stated in the introduction with

\[q= -t_{j}f^{'}(u_{m}),\quad g=t_{j}[f(u_{m}) + f^{'}(u_{m})u_{m}], \quad \mbox{ and } \quad v=u_{m+1}.\]

Initially, a weak solution in $H_{0}^{1}$ is desired, so it makes sense that $u$ be in $H_{0}^{1}$ and that $f$ and $f'$ should be defined on $H_{0}^{1}$.  However, as will be shown in Section 6, an $H_{0}^{1}$ solution to (**) is also in $H^{2}$. In light of this, $f$ and $f'$ need only be defined on $H^{2}$.  Note that if $f$ maps $H^{2}$ to $L^{2}$ and $f'$ maps $H^{2}$ to $L^{n}$, then $g$ is in $L^{2}$ for all dimensions $n>2$, via the Sobolev imbedding theorem.  Indeed, since $u$ is in $H^{1}$, $u$ is again in $L^{\frac{2n}{n-2}}$ and the H\"older inequality gives

\[\int_{\Omega}{[f'(u)u]^{2}} \leq C||f'(u)||_{L^{n}}^{2}||u||_{L^{\frac{2n}{n-2}}}^{2}. \]

\noindent To fullfill the positivity condition on $q$ in (**), we impose that $-f'>0$.  Now, at each time $t_{j}>0$ and for all $m$, the $m^{th}$ step in the iteration at time $t_{j}$ is a model for (**). \\ 

For the remainder of the article, we assume $\Omega$ is a bounded domain in $\R^{n>2}$ with smooth boundary $\Gamma$ and make the following assumptions (I)-(IV) on the nonlinear function $f$:  

\begin{itemize}

\item I.   $f$  is a continuous map from  $H^{2}(\Omega)$ to $L^{2}(\Omega).$
\item II.  $f'$ and $f''$ are continuous maps from  $H^{1}(\Omega)$ to $L^{n}(\Omega)$
\item III. there exists a constant $M>0$ such that

 \[|f|\leq M, \hspace{3mm} |f'|\leq M,\hspace{3mm}  \mbox{  and  } \hspace{3mm} |f''|\leq M.\]

\item IV.  $(-f')>0$.

\end{itemize}

Remark:  There is a redundancy and lack of `sharpness' in assumptions (I) and (II), given (III). Indeed, if the functions $f$, $f'$, and $f''$ are bounded, they naturally map to bounded functions on $\Omega$, and hence to $L^{\infty}(\Omega)$ which is contained in $L^{p}(\Omega)$ for all $p\geq 1$ since $\Omega$ is bounded. The reason for stating $L^{2}$ explicitly is that it is a \textsl{familiar} assumption for framing weak solutions to linear elliptic problems.  The bounds on the functions are not necessary to existence and uniqueness in (**), nor to the regularity lifting of the $H_{0}^{1}$ solution to $H^{2}$ .  Moreover, the $L^{2}$ hypothesis on $f$ and the $L^{n}$ hypothesis on $f'$ are sufficient for existence and uniqueness and the regularity lifting.  For a more general treatment of elliptic equations with measurable coefficients, see \cite{Trud}.

\section{Existence and Uniqueness}

For this Section, we assume (I), (II), and (IV).  To prove existance and uniqueness for (**) in $H_{0}^{1}(\Omega)$ ($H^{1}$ functions with zero on the boundary), the Riesz Representation theorem is sufficient.  We seek a unique solution in $H_{0}^{1}(\Omega)$.  The associated energy form for (**) is 

\[B(u,v)=\int_{\Omega}{DuDv + quv}.\]

\noindent It is well defined on $H_{0}^{1}(\Omega)$.  Indeed, since $n>2$ and $u,v\in H_{0}^{1}(\Omega)$, then $u,v\in L^{\frac{2n}{n-2}}(\Omega)$ by the Sobolev imbedding theorem.
Also since $\Omega$ is bounded, if $q\in L^{n}(\Omega)$, then $q\in L^{\frac{n}{2}}(\Omega)$.  Note that

\[\frac{2}{n} + \frac{n-2}{2n} + \frac{n-2}{2n}=1. \]

\noindent Therefore by H\"older's inequality, $quv$ is integrable over $\Omega$ with

\[\int_{\Omega}{|quv|} \leq ||q||_{L^{\frac{n}{2}}}||u||_{L^{\frac{2n}{n-2}}}||v||_{L^{\frac{2n}{n-2}}}.\]

\noindent This inequality combined with the Sobolev inequality

\[\||u||_{L^{\frac{2n}{n-2}}} \leq C||u||_{H_{0}^{1}}\]

\noindent gives

\[|B(u,v)|\leq C||u||_{H_{0}^{1}}||v||_{H_{0}^{1}} \]

\noindent where $C>0$ is dependent on $\Omega$, $n$, and $||q||_{L^{\frac{n}{2}}}$ but not on $u$ and $v$.  By the Poincar\'e inequality and the positivity of $q$, we have

\[||u||_{H_{0}^{1}}^{2} \leq C\int_{\Omega}{|Du|^{2}} \leq C\int_{\Omega}{|Du|^{2} + qu^{2}}=CB(u,u) \]

\noindent where $C>0$ is dependent on $n$ and $\Omega$ but not on $u$.  Since $f\in L^{2}(\Omega)$, it is a bounded linear functional on $H_{0}^{1}(\Omega)$ \cite{Evans}.  Since $B(u,v)$ is an inner product on $H_{0}^{1}$, the Riesz Representation theorem provides a unique $u^{*}\in H_{0}^{1}(\Omega)$ such that

\[B(u^{*},v)= \int_{\Omega}{fv} \hspace{10mm} \mbox{ for all } v\in H_{0}^{1}. \]

\noindent In other words, $u^{*}$ is the unique weak solution to (**) in $H_{0}^{1}$.  
 
\section{Regularity}

With the same hypotheses as in the previous Section, we wish to lift the regularity of the unique solution to (**) from $H_{0}^{1}$ to $H^{2}$, with the estimate controlled by the $L^{2}$ norm of $g(x)$.  Theorem 6.3.4 (Boundary $H^{2}$-regularity) in \cite{Evans} gives the desired regularity lifting of a solution to (**) when $q\in L^{\infty}$.  However, the $L^{\infty}$ condition is only used in factoring out $||q||_{L^{\infty}}$ from the following integral to find, for $u,v\in H^{1}$ and $\epsilon >0$ in Cauchy's inequality,

\[\int{|quv|}\leq ||q||_{L^{\infty}}\int{|uv|}\leq C\left(\frac{1}{2\epsilon}||u||_{L^{2}}^{2} + \frac{\epsilon}{2}||v||_{L^{2}}^{2}\right). \]

\noindent The  $L^{n}$ hypothesis on $q$ provides,

\[\int{|quv|}\leq \frac{1}{2\epsilon}||cu||_{L^{2}}^{2} + \frac{\epsilon}{2}||v||_{L^{2}}^{2}\leq \frac{1}{2\epsilon}\left(||q||_{L^{n}}^{2}||u||_{L^{\frac{2n}{n-2}}}^{2}\right) + \frac{\epsilon}{2}||v||_{L^{2}}^{2}\]

\[\leq C\left(\frac{\epsilon}{2}||u||_{H^{1}}^{2} + \frac{\epsilon}{2}||v||_{L^{2}}^{2}\right)\leq C\left(\frac{\epsilon}{2}||Du||_{L^{2}}^{2} + \frac{\epsilon}{2}||v||_{L^{2}}^{2}\right) \]

\noindent by H\"older's inequality, the Sobolev imbedding theorem and Poincare's inequality. By the above estimates, we have also

\[\int{(cu)}^{2}\leq M||Du||_{L^{2}}^{2}. \]

\noindent Following the line of reasoning in \cite{Evans}, the result for $q\in L^{n}$ is a sufficient replacement for the estimate for $L^{\infty}$ to get the regularity estimate, 

\[||u||_{H^{2}(\Omega)} \leq C\left(||g||_{L^{2}(\Omega)} + ||u||_{H^{1}(\Omega)}\right) \]

\noindent where $C$ depends only on $\Omega$ and $n$ and $q$.  Now, recalling the second energy estimate above,  

\[||u||_{H^{1}(\Omega)}^{2}(\Omega) \leq CB(u,u) = C\int_{\Omega}{gu} \leq C\left(\frac{1}{2}||g||_{L^{2}(\Omega)}^{2} + \frac{1}{2}||u||_{L^{2}(\Omega)}^{2}\right). \]

\noindent since $u$ is a weak solution to (**).  The last inequality is given by Cauchy's inequality with $\epsilon=1$. Also since $u$ is a unique solution, the $L^{2}$ norm of $u$ is controlled by the $L^{2}$ norm of $g$ by Theorem 6.2.6 in \cite{Evans}.  Therefore,

\[||u||_{H^{2}(\Omega)} \leq C||g||_{L^{2}(\Omega)},\]

\noindent where $C$ depends only on $\Omega$, $n$, and more \textsl{significantly}, $q$.\\

\noindent To summarize the results in Sections 5 and 6, we have:\\

\begin{thm}{\bf 6.1}  \textsl{Let $\Omega$ be a bounded domain in $\R^{n}$ with smooth boundary $\Gamma$ and $n>2$. Then for $g\in L^{2}(\Omega)$, $q\in L^{n}(\Omega)$, and $q>0$, the linear boundary value problem}

\[(**)
\left\{ \begin{array}{rcl}
-\Delta u + q(x)u&=&g(x)   \hspace{5mm} \mbox{ in } \Omega\\
u|_{\Gamma}&=& 0  \hspace{5mm} \mbox{ on } \Gamma,
\end{array}
\right.
\]

\noindent \textsl{has a unique solution $u\in H^{2}(\Omega)\bigcap H_{0}^{1}$ with} 

\[||u||_{H^{2}(\Omega)} \leq C(||g||_{L^{2}(\Omega)}),\]

\noindent \textsl{where $C$ depends only on $\Omega$, $n$, and $q$.}

\end{thm}

\section{Convergence} 

In the previous two Sections, it was shown that (**) is uniquely solvable in $H^{1}$ and the solution is \textsl{a priori} in $H^{2}$ with estimate controlled by the forcing term $g$.  Recalling that (**) represents an arbitrary iteration of Newton's method at time $t_{j}$,  the linear equation solved by the difference, $u_{m+1}-u_{m}$ for $m>1$, is given by

\[-\Delta (u_{m+1}-u_{m})+ (-t_{j}f'(u_{m}))(u_{m+1}-u_{m})\] \[=t_{j}(f(u_{m})-f(u_{m-1})-f'(u_{m-1})(u_{m}-u_{m-1})) \hspace{7mm} \mbox{ in } \Omega\]

$\displaystyle{\hspace{39mm} u_{m+1}-u_{m}=0 \hspace{29mm} 
\mbox{ on } \Gamma.}$

\noindent This is (**) with 

\[v=u_{m+1}-u_{m},\hspace{3mm} q=-t_{j}f'(u_{m}),\]  

\[g=t_{j}(f(u_{m})-f(u_{m-1})-f'(u_{m-1})(u_{m}-u_{m-1})),\] 

\noindent and a zero boundary condition.  Indeed, using the same argument as at the end of Section 3, it is clear that $g\in L^{2}$.  For $m=0$, by the definition of $u_{0}$ at time $t_{j}$, the problem satisfied by $u_{1}-u_{0}$ is

\[-\Delta (u_{1}-u_{0})+ (-t_{j}f'(u_{0}))(u_{1}-u_{0})\] \[=(t_{j}-t_{j-1})f(u_{0}) \hspace{7mm} \mbox{ in } \Omega\]

$\displaystyle{\hspace{39mm} u_{1}-u_{0}=0 \hspace{12mm} 
\mbox{ on } \Gamma,}$\\

\noindent and is again a model for (**).  To facilitate the convergence estimates to follow, it will be helpful to use Taylor's theorem to simplify $g$.  Similar to the application of a mean value theorem used in \cite{Hsiao}, for $m>1$, $g$ can be written as

\[g=t_{j}(u_{m}-u_{m-1})^{2}\int_{(0,1)}{f''(\tau u_{m}+(1-\tau)u_{m-1})(1-\tau)d\tau}. \]

Theorem 6.1 and the boundedness of $f''$ give the estimate,

\[||u_{m+1}-u_{m}||_{H^{2}} \leq C||t_{j}(u_{m}-u_{m-1})^{2}\int_{(0,1)}{f''(\tau u_{m}+(1-\tau)u_{m-1})(1-\tau)d\tau}||_{L^{2}} \]

\[\leq \frac{Ct_{j}M}{2}||(u_{m}-u_{m-1})^{2}||_{L^{2}} \]

\[\leq \frac{Ct_{j}M}{2}||(u_{m}-u_{m-1})||_{L^{4}}^{2}. \]

Before progressing with the estimate, it is important to discuss the dependence on dimension.  For dimensions $n=3$ and $n=4$, the $L^{4}$ norm is controlled by the $H^{1}$ norm, by the Sobolev imbedding theorem, which in turn is controlled by the $H^{2}$ norm. For dimensions $n=5,6,7,$and $8$, the $L^{4}$ norm is controlled by the $H^{2}$ norm, via the more general Sobolev inequality [1,p.270].  The subsequent calculations do not depend on which dimension $n\in (3,4,5,6,7,8)$ is assumed.  However, only in dimension $n=3$ does the general Sobolev theorem assure that our $H^{2}$ solution is indeed continuous.  For $n=5,6,7,$ and $8$, the $H^{2}$ solution is respectively, $L^{10}$, $L^{6}$, $L^{\frac{14}{3}}$, and $L^{4}$.  To continue with the convergence estimate, for $n\in (3,4,5,6,7,8)$, we have

\[\frac{Ct_{j}M}{2}||(u_{m}-u_{m-1})||_{L^{4}}^{2} \leq \frac{Ct_{j}MC_{s}}{2}||(u_{m}-u_{m-1})||_{H^{2}}^{2}\]

\noindent where $C_{s}$ is the constant from the Sobolev theorem and only depends on $\Omega$ and $n$.  Since in Theorem 6.1, $C$ depends on $||f'(u_{m}(x,t_{j}))||_{L^{n}}$ and hence $m$ and $t_{j}$, we invoke the boundedness of $f'$. Therefore $||f'(u_{m}(x,t_{j}))||_{L^{n}}$ is bounded by some constant $C>0$, uniformly over $m$ and $t_{j}$.  Let $K=\frac{CMC_{s}}{2}$.  Inductively,

\[||u_{m+1}-u_{m}||_{H^{2}} \leq (t_{j}K||u_{1}-u_{0}||_{H^{2}})^{2^{m}-1}||u_{1}-u_{0}||_{H^{2}}\]

\noindent and therefore for $s\in \mathbb{N}$,

\[||u_{m+s}-u_{m}||_{H^{2}}\leq [a^{2^{m+s-1}-1}+...+a^{2^{m}-1}]||u_{1}-u_{0}||_{H^{2}} \]

\noindent where $a=t_{j}K||u_{1}-u_{0}||_{H^{2}}$.  If $t_{j}$ is chosen such that $a<1$, then the $\textsl{positive}$ expression in brackets above is bounded from above by the tail end of a convergent geometric series, and therefore goes to zero as $m\rightarrow \infty$.  We have now shown that $u_{m}$ is a Cauchy sequence in the Banach space $H^{2}(\Omega)$, and therefore converges to some $u^{*}\in H^{2}(\Omega)$.  As stated in \cite{Hsiao}, due to the continuity of $f$ and the boundedness of $f'$, it is clear that $u^{*}$ satisfies

\[(*')
\left\{ \begin{array}{rcl}
-\Delta u &=&t_{j}f(u)   \hspace{5mm} \mbox{ in } \Omega\\
 u|_{\Gamma}&=& 0 \hspace{5mm} \mbox{ on } \Gamma
\end{array}
\right.
\]

\noindent almost everywhere and that the uniqueness of the solution $u^{*}$ follows from the uniquenss of the solution $u_{m}(x,t_{j})$ to (**) for each $m$ and $t_{j}$.  One additional assumption is necessary for $t_{j}$ to be chosen as above, as well as for progressing to $t=1$ in finitely many applications of Newton's method.  Assumption (V) will be a condition on the width of the time intervals $t_{j}-t_{j-1}$.  To make this precise we look at the the problem satisfied by $u_{1}-u_{0}$ at time $t_{j}$ and apply Theorem 6.1 and the boundedness of $f$ and $f'$ to estimate,

\[||u_{1}-u_{0}||_{H^{2}}\leq C||(t_{j}-t_{j-1})f(u_{0})||_{L^{2}}\]

\[\leq C(t_{j}-t_{j-1})||f(u_{0})||_{L^{2}} \leq MC(t_{j}-t_{j-1}).\]

\noindent If $A=MC$, then $A$ depends on the bounds on $f$ and $f'$, the volume of $\Omega$, and $n$, but not on $t_{j}$. In the following inequality,

\[Kt_{j}||u_{1}-u_{0}||_{H^{2}}\leq KAt_{j}(t_{j}-t_{j-1}) <1, \]

\noindent the condition for convergence was that the leftmost expression be $<1$.  Since $t_{j}\leq 1$ for all $j$, it suffices to make the assumption (V):

\begin{itemize}

\item V.   For each $j\geq 1$, \hspace{3mm} $t_{j}-t_{j-1} < \frac{1}{KA}$

\end{itemize}

\noindent As $KA$ only depends on $\Omega$, $p=2$, $n$, and $M$, (and in particular, not $j$), $KA$ gives a uniform bound on the time intervals, and therefore $t=1$ is attainable after finitely many applications of Newton's method. When $\Omega$ is a domain in $\R^{3}$, the $H^{2}$ solution is then continuous by the general Sobolev imbedding theorem.  We now list assumptions (I)-(V) and state the main result.

\begin{itemize}

\item I.   $f$  is a continuous map from  $H^{2}(\Omega)$ to $L^{2}(\Omega).$
\item II.  $f'$ and $f''$ are continuous maps from  $H^{1}(\Omega)$ to $L^{n}(\Omega)$
\item III. there exists a constant $M>0$ such that

 \[|f|\leq M, \hspace{3mm} |f'|\leq M,\hspace{3mm}  \mbox{  and  } \hspace{3mm} |f''|\leq M.\]

\item IV.  $(-f')>0$,
\item V.   For each $j\geq 1$, \hspace{3mm} $t_{j}-t_{j-1} < \frac{1}{KA}$

\end{itemize}

\begin{thm}{\bf 7.1} \textsl{With $\Omega$ a bounded domain in $\R^{3}$ with smooth boundary and assumptions (I)-(V),
the semilinear boundary value problem},

\[(*')
\left\{ \begin{array}{rcl}
-\Delta u&=&f(u)   \hspace{5mm} \mbox{ in } \Omega\\
u|_{\Gamma}&=& 0 \hspace{5mm} \mbox{ on } \Gamma=\partial \Omega,
\end{array}
\right.  
\]

\noindent has a unique solution in $H^{2}(\Omega)\bigcap H_{0}^{1}(\Omega)$, and hence a continuous solution, which can be approximated by the Newton-imbedding method.\\
\end{thm}
 
\section{Conclusion}
 The goal for improving this procedure is to weaken the assumptions on $f$ and $f'$.  In particular, to eliminate the boundedness or equivalently the uniform boundedness of $f(u)$ and $f'(u)$.  To do this requires a function $f$ such that $f(u_{m}(x,t))$ does not grow too fast in $L^{2}$ norm as $t$ increases and such that $f'(u_{m}(x,t))$ does not grow too fast in $L^{n}$ norm as $m$ and $t$ increase.  If the boundedness of $f$ is dropped from the assumptions, a linear function would be allowed, but assumption (IV) would force it to be decreasing.  Since the spectrum of $-\Delta$ is positive, (*') is then solved uniquely with $u\equiv 0$ (which is achieved vacuously in the procedure).  An example of a function satisfying (I)-(IV) is 
 
\[f(x)=cot^{-1}(x)\]

\noindent whose derivatives are 

\[f'(x)=\frac{-1}{1+x^{2}} \hspace{3mm} \mbox{  and  } \hspace{3mm} f''(x)=\frac{2x}{(1+x^{2})^{2}}. \]  

\noindent Similarly, if $\epsilon>0$, $A>0$, and $h,k\in \R$, then

\[Acot^{-1}\left(\frac{x-h}{\epsilon}\right)+k \]

\noindent represents a family of functions, each of which satify (I)-(IV).  A subset of this family, given by

\[f_{\epsilon}(x)=\frac{1}{\pi}cot^{-1}(\frac{x}{\epsilon}) -1,\] 

\noindent is of interest since 

\[f_{\epsilon}(x)\rightarrow -H \hspace{5mm} \mbox{ as } \hspace{5mm} \epsilon\rightarrow 0\]

\[f_{\epsilon}'(x)=\frac{-\epsilon}{\epsilon^{2}+x^{2}}\rightarrow -\delta \hspace{5mm} \mbox{ as } \hspace{5mm} \epsilon\rightarrow 0,\]

\noindent where $H$ is the Heaviside function and $\delta$ is the Dirac delta function and the arrows imply at least pointwise convergence and possibly a more refined limit.  It is natural to ask whether the Newton-imbedding procedure can be carried out in a distributional setting with $f=-H$ and whether  $f_{\epsilon}$ produces a meaningul approximation to the Heaviside function for small $\epsilon$.  More generally, if $\mathcal{P}$ is the class of functions which satisfy (I)-(IV), it is of interest as to which functions exist in a suitable closure of $\mathcal{P}$.
In this case, `suitable closure' can be taken to mean one whose functions allow for the application of the Newton-imbedding procedure in possibly a distributional or more general setting, and produce a solution which can be approximated by applying the procedure to a function in $\mathcal{P}$.  

\section*{Acknowledgement}

First, I would like to thank Professor Congming Li at the University of Colorado, Boulder for discussion and encouragement.  In addition I am grateful for the feedback from Professor James P. Kelliher at the University of California, Riverside and John Huerta at the University of California, Riverside.  Finally I would like to thank my advisor, Professor Michel L. Lapidus, for continual support in all of my mathematical endeavors.


\begin{thebibliography}{999}
\bibitem{Evans}
Evans, Lawrence C., $\textsl{Partial Differential Equations}$, Graduate Studies in Mathematics \textbf{19}, Amer. Math. Soc., Providence, RI, (1998).
 
\bibitem{Hsiao}
Hsiao, George C., A Newton-imbedding procedure for solutions of semilinear boundary value problems in Sobolev spaces, Complex Variables and Elliptic Equations, Nos.8-11, \textbf{51} 1021-1032, (2006). 

\bibitem{Trud}
Neil S. Trudinger, Linear elliptic operators with measurable coefficients, \textsl{Ann. Scuola. Norm. Sup. Pisa, Sci. Fis. Mat.}, \textbf{27}, No.3, 265-308, (1973).

\end{thebibliography}
\end{document}